\documentclass{elsarticle}

\usepackage[english]{babel}
\usepackage{amssymb,amsmath,amsthm}

\usepackage{hyperref}

\usepackage{booktabs}

\usepackage{thm-restate}
\usepackage{thmtools}
\declaretheorem[name=Theorem, numberwithin=section]{thm}

\usepackage{tikz}
\usetikzlibrary{calc}

\newcommand{\ml}{\mathrm{ml}}

\newtheorem{theorem}{Theorem}[section]
\newtheorem{definition}[theorem]{Definition}
\newtheorem{claim}[theorem]{Claim}
\newtheorem{lemma}[theorem]{Lemma}
\newtheorem{corollary}[theorem]{Corollary}
\newtheorem{conjecture}[theorem]{Conjecture}
\newtheorem{remark}[theorem]{Remark}

\begin{document}

\begin{frontmatter}
\title{On the minimum leaf number of cubic graphs}
\author[ug]{Jan Goedgebeur\fnref{fn2}}
\ead{jan.goedgebeur@ugent.be}
\author[ynu]{Kenta Ozeki\fnref{fn1}}
\ead{ozeki-kenta-xr@ynu.ac.jp}
\author[ug]{Nico Van Cleemput}
\ead{nico.vancleemput@gmail.com}
\author[bute]{G\'abor Wiener\fnref{fn3}}
\ead{wiener@cs.bme.hu}

\fntext[fn2]{Supported by a Postdoctoral Fellowship of the Research Foundation Flanders (FWO).}
\fntext[fn1]{Partially supported by JST ERATO Grant Number JPMJER1201, Japan, and JSPS KAKENHI Grant Number 18K03391.}
\fntext[fn3]{Research was supported by grants 108947 and 124171 of the National Research, Development and Innovation Office --  NKFIH.}

\address[ug]{Department of Applied Mathematics, Computer Science and Statistics, Ghent University, Krijgslaan 281-S9, 9000 Ghent, Belgium.}
\address[ynu]{Faculty of Environment and Information Sciences, Yokohama National University, 79-2 Tokiwadai, Hodogaya-ku, Yokohama 240-8501, Japan.}
\address[bute]{Department of Computer Science and Information Theory, Budapest University of Technology and Economics, Hungary.}

\begin{abstract}
The \emph{minimum leaf number} $\hbox{ml} (G)$ of a connected graph $G$ is defined as the minimum number of leaves of the spanning trees of $G$. We present new results concerning the minimum leaf number of cubic graphs: we show that if $G$ is a connected cubic graph of order $n$, then $\ml(G) \leq \frac{n}6 + \frac13$, improving on the best known result in~\cite{sala} and proving the conjecture in~\cite{ZY}. We further  prove that if $G$ is also 2-connected, then $\ml(G) \leq \frac{n}{6.53}$, improving on the best known bound in~\cite{BSSS}. We also present new conjectures concerning the minimum leaf number of several types of cubic graphs and examples showing that the bounds of the  conjectures are best possible. 
\end{abstract}


\end{frontmatter}

\section{Preliminaries}  

All graphs considered in this paper are finite, simple, and connected, unless stated otherwise. For a graph $G$, $V(G)$ and $E(G)$ denotes the set of vertices and the set of edges of $G$, respectively. 
The subgraph of $G$ induced by the vertex set $X\subseteq V(G)$ is denoted by $G[X]$.
A graph is said to be traceable if it has a hamiltonian path.     

\begin{definition}
The minimum leaf number of a connected graph $G$, denoted by $\hbox{ml} (G)$ is the minimum number of leaves of the spanning trees of $G$. 
\end{definition}

\begin{definition}
The path covering number of a graph $G$, denoted by $\mu (G)$ is the minimum number of vertex disjoint paths that cover the vertices of $G$. 
\end{definition}

\begin{claim} \label{all1}
$\mu (G) +1 \leq \ml (G) \leq 2\mu (G) $.
\end{claim}

\begin{proof}
If $G$ has a spanning tree with $k\geq 3$ leaves, then we can disjointly cover the vertices of $G$ by a path and a tree with $k-1$ leaves, from which the first inequality immediately follows by induction. Indeed, let $T$ be the spanning tree with $k$ leaves, $a$ be one of the leaves, $b$ that branch (that is, vertex of degree 3) of $T$, which is the closest to $a$ in $T$, and finally $b'$ be the neighbour of $b$ along the unique $ba$ path in $T$. Then if we delete the vertices of the $ab'$ path from $T$, we obtain a tree with $k-1$ leaves (since $b$ is a branch). The second inequality is even easier to prove: $G$ has a path cover $S$ consisting of  $\mu (G)$ paths. $S$ is a forest, where the number of vertices of degree at most 1 (note that some paths of $S$ might be isolated vertices) is at most $2\mu (G)$. By extending the forest $S$ to a spanning tree arbitrarily, we obtain a spanning tree with at most $2\mu (G)$ leaves.  
\end{proof}

Both inequalities of Claim~\ref{all1} are tight, even for 2-connected cubic graphs, as we shall see later. In Section~\ref{sect:proof_main_results} we shall also see that for 2-connected cubic graphs on $n$ vertices  $\ml (G) = 2\mu (G) $ is possible only if $\mu (G) \leq \frac{n}{18} $.

\section{Introduction}  

Hamiltonian properties of cubic planar graphs have been studied extensively due to Tait's attempt to prove the four color conjecture based on the conjecture that all 3-connected cubic planar graphs are hamiltonian. Tait's conjecture was disproved by Tutte~\cite{Tu} in 1946, 
his counterexample is of order 46; the smallest counterexample (of order 38) was found in 1966 by Barnette, Bos\'ak, and Lederberg, independently~\cite{Bo,Le}, see also~\cite{Gb}. In 1986, Holton and McKay~\cite{HM} proved that there exists no non-hamiltonian $3$-connected cubic planar graph on fewer than 38 vertices. In the late sixties Barnette conjectured two weakened versions of Tait's conjecture, of which the more famous one stating that all 3-connected bipartite planar cubic graphs are hamiltonian~\cite{Ba} is still open. The other one 
(conjectured also by Goodey~\cite{Go}) stating that all 3-connected cubic plane graphs with faces of size at most 6 are hamiltonian has been proved recently by Kardo\v{s}~\cite{Kardos}. In 1982, Asano, Exoo, Harary, and Saito~\cite{AEHS} showed that the (unique) smallest 2-connected cubic bipartite planar non-hamiltonian graph has order 26.

Hamiltonian properties of cubic, but not necessarily planar graphs have also been investigated, though not so extensively.
In 1971 Tutte conjectured that all 3-connected bipartite cubic graphs are hamiltonian, which is a (much)
stronger version of the still open Barnette-conjecture. Actually, Tutte conjectured even more: that the hamiltonian cycles of such a graph $G$ span the cycle space of the graph (and could prove that the 2-factors indeed span the cycle space)~\cite{Tu3}. This conjecture was disproved by Horton (see \cite[p. 240.]{BM}).

Exoo and Harary (continuing the work in~\cite{AEHS}) described the smallest examples of several non-hamiltonian graphs 
and multigraphs~\cite{EH82}, both in the planar and the nonplanar case.  

The most natural optimization type generalizations of hamiltonicity and traceability are the path covering number and the minimum leaf number. The path covering number of cubic graphs is studied in the famous paper of Reed~\cite{Re}, who showed that if $G$ is a connected cubic graph  of order $n$, then $\mu(G) \leq \lceil \frac{n}9 \rceil$ ~\cite{Re}. Reed also conjectured 
that for 2-connected cubic graphs $\mu(G) \leq \frac{n}{10}$. This has been recently confirmed by Yu~\cite{Yu}. It is worth mentioning that 
for not necessarily connected graphs the bound is much weaker: Magnant and Martin~\cite{MM} conjectured that for $k$-regular, not necessarily connected graphs $G$, $\mu (G) \leq \frac{n}{k+1}$ and proved the conjecture for $k\leq 5$ (while by Dirac's well-known theorem~\cite{Di} it is also true for $k\geq \frac{n-1}2$).

Somewhat less is known about the minimum leaf number. Zoeram and Yaqubi~\cite{ZY} conjectured in 2017 that if $G$ is a cubic graph  of order $n$, then $\ml (G) 
\leq \frac{n}6 + \frac13$, and gave an example that this bound (if true) is sharp. Actually this conjecture was almost proved in a much earlier paper of Salamon and Wiener~\cite{sala}, who showed that a refined version of depth first search finds a spanning tree of a connected cubic graph with at most $\frac{n}6 + \frac43$  leaves. In Section~\ref{res} we improve this result and thereby prove the Zoeram-Yaqubi conjecture.  The example given by Zoeram and Yaqubi has connectivity 1, so better bounds can still be possible for graphs with a higher connectivity.

Boyd, Sitters, van der Ster, and Stougie~\cite{BSSS} -- based on a result of 
M\"omke and Svensson~\cite{MS} -- showed that if $G$ is a 2-connected cubic multigraph of order $n$, then $\ml (G) \leq \frac{n}6 + \frac23$. 
(For cubic multigraphs, without the 2-connectivity only $\ml (G) \leq \frac{n}4 + \frac12$ is true. The proof -- based on the fact that the leaves of an optimum spanning tree must form an independent set -- is straightforward, and examples showing that the bound is tight are easy to find).

Finally let us mention some results concerning spanning trees with a high number of leaves: Storer~\cite{St} showed in 1981 that all cubic graphs  of order $n$ have a spanning tree with at least $\frac{n}4 +2$ leaves (the bound is tight). Griggs, Kleitman, and Shastri~\cite{GKS} proved that all 3-connected cubic graphs  of order $n$ have a spanning tree with at least $\frac{n}3 +\frac43$ leaves (this bound is also tight) and they gave a rigorous proof of Storer's theorem.  Kleitman and West~\cite{KW} generalized Storer's result to graphs with minimum degree 3.

\section{Results and conjectures} \label{res}

\subsection{General case}

The main result of the paper is the following.


\begin{restatable}{thm}{mainthma}\label{fo}
If $G$ is a 2-connnected cubic graph of order $n$ then $\ml (G) \leq \frac{13}{85} n \approx \frac{n}{6.538}$.
\end{restatable}

As a consequence of Theorem~\ref{fo} we also prove the conjecture of Zoeram and Yaqubi~\cite{ZY}, that is


\begin{restatable}{thm}{mainthmb}\label{al}
If $G$ is a connected cubic graph  of order $n$ then $\ml (G) \leq  \frac{n}{6} + \frac{1}{3}$.
\end{restatable}

The proofs of these results can be found in the next section. In what follows we formulate some conjectures concerning the minimum leaf number of certain classes of cubic graphs. We also show that these conjectures (if true) are best possible. 

\begin{conjecture} \label{sej1}
If $G$ is a 2-connnected cubic graph  of order $n$ then $\ml (G) \leq \lceil \frac{n}{10} \rceil$. 
\end{conjecture}

This conjecture might seem far-fetched, given that Reed~\cite{Re} claims that there exist 2-connnected cubic graphs $G$ with $\mu (G) \geq \lceil \frac{n}{10} \rceil$ and by Claim~\ref{all1} $\ml (G) \geq \mu (G) +1 $ always holds. However, Reed's example only gives graphs with $\mu (G) \geq \lceil \frac{n}{20} \rceil$, as pointed out by Yu~\cite{Yu}. Yu also gives 2-connnected cubic graphs $G$ with $\mu (G) \geq \lceil \frac{n}{14} \rceil$ in Theorem 1.2 of~\cite{Yu} and this is the best known bound to date (later we prove that Yu's graphs have path covering number \emph{exactly} $\frac{n}{14}$). 

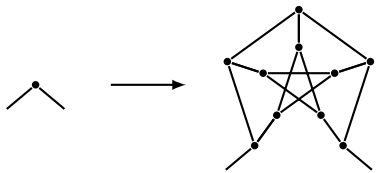
\begin{figure}
\begin{center}
\begin{tikzpicture}[thick,every node/.style={fill, circle, inner sep=1pt}]
\begin{scope}
\node (a) at (0,0) {};
\draw (a) -- ++(220:.5) (a) -- ++(320:.5);
\end{scope}
\begin{scope}[shift={(3.5,0)}]
\foreach \x/\angle in {1/306,2/18,3/90,4/162,5/234} {
	\node (outer\x) at ($(0,0) + (\angle:1)$) {};
	\node (inner\x) at ($(0,0) + (\angle:.5)$) {};
}
\def\prevX{1}
\draw \foreach \x[remember=\x as \prevX] in {2,...,5} {(outer\prevX) -- (outer\x) -- (inner\x)};
\def\prevX{4}
\draw \foreach \x[remember=\x as \prevX] in {1,3,5,2,4} {(inner\prevX) -- (inner\x) -- (outer\x)};
\draw (outer5) -- ++(220:.5) (outer1) -- ++(320:.5);
\end{scope}

\draw[-latex] (1,0) -- (2,0);
\end{tikzpicture}
\end{center}
\caption{An operation to construct a non-traceable 2-connected cubic graph from a cycle.}\label{fig:gadget_general_2conn}
\end{figure}

In order to prove that Conjecture~\ref{sej1} is sharp, let $G$ be the graph obtained from a cycle of length $k$ by substituting all vertices of the cycle by an edge-deleted Petersen graph $P'$ (that is, $P'$ is a graph obtained from the Petersen graph by deleting an edge), such that the edges of the cycle are connected to the vertices that are incident with the deleted edge of the respective copy of $P'$ (see Figure~\ref{fig:gadget_general_2conn}). $G$ is easily seen to be cubic and of connectivity 2. Let $T$ be an arbitrary spanning tree of $G$, $X$ be the vertex set of one of the 
copies of $P'$ and $a,b\in X$ the degree 2 vertices of $G[X]$. We claim that $T[X]$ has a vertex $v$ of degree 1, such that $v\neq a$, $v\neq b$. This clearly implies $\ml (G) \geq \lceil \frac{|V(G)|}{10} \rceil$, since $v$ has the same degree in $T$ and $T[X]$, hence there must be $k$ vertices of degree 1 in $T$, while the order of $G$ is $10k$. The proof of the claim is not difficult: $T[X]$ is either a tree or a forest with two components, such that $a$ and $b$ are in different components. If $T[X]$ is a path, one of the endvertices must be different from $a$ and $b$, otherwise there would be a hamiltonian cycle in the Petersen graph. If $T[X]$ is a tree, but not a path, then the claim is obvious. Finally, if $T[X]$ consists of two trees, at least one of these trees must contain a leaf different from $a$ and $b$ (since $a$ and $b$ are in different trees and at least one of the trees must contain at least one edge). 

\begin{remark} 
It is not difficult to verify that for the graphs $G$ constructed above we have $\mu (G) = \frac{n}{20}$, thus the bound $\ml (G) \leq 2 \mu (G)$ of Claim \ref{all1} is sharp, even for 2-connected cubic graphs. 
\end{remark} 

\begin{remark} 
It is worth mentioning that the ceiling sign in the conjecture is needed, since there exist non-traceable 2-connected (moreover, also 3-connected or 2-connected and planar) cubic graphs of order 28, as we shall see later. 
\end{remark} 

\begin{conjecture} \label{sej2}
If $G$ is a 3-connected cubic graph  of order $n$ then $\ml (G) \leq \lceil \frac{n}{16} + \frac12 \rceil$.
\end{conjecture}

To verify that this conjecture is also sharp, let us consider the 3-connected graphs $G_{k}$ appearing in the papers~\cite{Wi1,Wi2}. For the sake of completeness we describe these graphs here. First we need the following definitions.

\begin{definition}
A pair of vertices $(a,b)$ of a graph $G$ is said to be good if there exists a hamiltonian path of $G$ between them. A pair of pairs of vertices of $G$ $((a,b),(c,d))$ is said to be good if there exists a spanning subgraph of $G$ consisting of two vertex-disjoint paths, one between $a$ and $b$ and another one between $c$ and $d$. 
\end{definition}

\begin{definition} (Hsu, Lin \cite{HL})
The quintuple $(H,a,b,c,d)$ is a J-cell if $H$ is a graph and $a,b,c,d \in V(H)$, such that 
\begin{enumerate}
\item The pairs $(a,d)$, $(b,c)$ are good in $H$. 
\item None of the pairs $(a,b)$, $(c,d)$, $(a,c)$, $(b,d)$, $((a,b),(c,d))$, $((a,c),(b,d))$ are good in $H$.
\item For each $v\in V(H)$ there is a good pair in $H-v$ among  $(a,b)$, $(c,d)$, $(a,c)$, $(b,d)$, $((a,b),(c,d))$, $((a,c),(b,d))$. 
\end{enumerate}
\end{definition}

It is known that all J-cells can be obtained by deleting two adjacent vertices of degree 3 of some hypohamiltonian graph \cite{Wi2}, thus the smallest J-cell is obtained from the Petersen graph by deleting two adjacent vertices, see Figure~\ref{figJcell}. 

\begin{figure}[h!] 
\begin{center}
\begin{tikzpicture}[every node/.style={circle, fill, inner sep=1.5pt, thick},every edge/.style={draw, thick}]

\node (c) at (180:2) [label=left:$c$] {};
\node (a) at (90:2) [label=above:$a$] {};
\node (b) at (0:2) [label=right:$b$] {};
\node (d) at (270:2) [label=below:$d$] {};
\node (1) at (135:1)  {};
\node (2) at (45:1)  {};
\node (3) at (315:1)  {};
\node (4) at (225:1)  {};

\draw[thick] (c) -- (1) -- (a) -- (2) -- (b) -- (3) -- (d) -- (4) -- (c) (1) -- (3) (2) -- (4);
\end{tikzpicture}
\end{center}
\caption{The smallest J-cell (obtained from the Petersen graph)\label{figJcell}}
\end{figure}
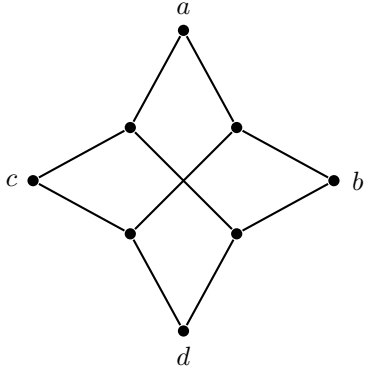
 
Now let $F_i=(H_i,a_i,b_i,c_i,d_i)$ be J-cells for $i=1,2,\ldots ,k$. and let us define the graphs $G_k$ as follows. $G_k$ consists of vertex-disjoint copies of the graphs $H_1, H_2, \ldots , H_k$, the edges $(b_i,a_{i+1}), (c_i,d_{i+1})$ for all $i=1,2,\ldots k-1$, and the edges $(b_k,a_{1}), (c_k,d_{1})$. Figure~\ref{figGk} shows the construction of $G_k$. 

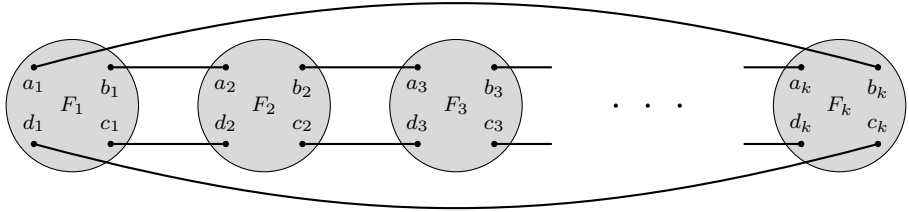
\begin{figure}[h!]
\begin{center}
\begin{tikzpicture}[ fo/.style={draw, circle, fill=black, minimum size={0.08cm}, inner sep=0cm, scale=0.88}, font=\footnotesize, scale=0.85]
\tikzstyle{b} = [draw, fill=black!15!white, minimum size={2cm}, circle, scale=0.88, font=\footnotesize, node distance=0.5cm]
\node[b] (1) at (0,0) [label=center:$F_1$] {};
\node [fo] (c) at (-0.6,-0.6) [label=above:${d}_1$] {};
\node [fo] (a) at (-0.6,0.6) [label=below:${a}_1$] {};
\node [fo] (d) at (0.6,-0.6) [label=above:${c}_1$] {};
\node [fo] (b) at (0.6,0.6) [label=below:${b}_1$] {};

\node[b] (2) at (3,0) [label=center:$F_2$] {};
\node [fo] (c2) at (3-0.6,-0.6) [label=above:${d}_2$] {};
\node [fo] (a2) at (3-0.6,0.6) [label=below:${a}_2$] {};
\node [fo] (d2) at (3+0.6,-0.6) [label=above:${c}_2$] {};
\node [fo] (b2) at (3+0.6,0.6) [label=below:${b}_2$] {};

\node[b] (3) at (6,0) [label=center:$F_3$] {};
\node [fo] (c3) at (6-0.6,-0.6) [label=above:${d}_3$] {};
\node [fo] (a3) at (6-0.6,0.6) [label=below:${a}_3$] {};
\node [fo] (d3) at (6.6,-0.6) [label=above:${c}_3$] {};
\node [fo] (b3) at (6.6,0.6) [label=below:${b}_3$] {};

\fill[black] (8.5,0) circle (0.033cm) ; \fill[black] (9,0) circle (0.033cm) ; \fill[black] (9.5,0) circle (0.033cm) ;

\node[b] (4) at (12,0) [label=center:$F_k$] {};

\draw[thick] (6.6,0.6) -- (7.5,0.6); \draw[thick] (6.6,-0.6) -- (7.5,-0.6); \draw[thick] (10.5,0.6) -- (11.4,0.6); \draw[thick] (10.5,-0.6) -- (11.4,-0.6); 
\node [fo] (ck) at (12-0.6,-0.6) [label=above:${d}_k$] {};
\node [fo] (ak) at (12-0.6,0.6) [label=below:${a}_k$] {};
\node [fo] (dk) at (12.6,-0.6) [label=above:${c}_k$] {};
\node [fo] (bk) at (12.6,0.6) [label=below:${b}_k$] {};

\draw[thick] (b) -- (a2)  (b2) -- (a3)   (d) -- (c2)  (d2) -- (c3);
\draw[thick, bend right=15] (bk) to (a);
\draw[thick, bend right=-15] (dk) to (c);
\end{tikzpicture}
\end{center}
\caption{Construction of the graphs $G_k$ \label{figGk}}
\end{figure}

It is proved in \cite{Wi1,Wi2} that $\ml (G_{2k+1}) = k+1$. Therefore, if the J-cells used in the construction are the 8-vertex J-cells obtained from the Petersen graph, then $G_{2k+1}$ is cubic and its order is $n=16k+8$ for any $k\geq 2$, showing that $\ml (G_{2k+1}) = \lceil \frac{n}{16} + \frac12 \rceil$, indeed.

\smallskip

We describe another sequence of graphs $G$, for which we have the somewhat weaker lower bound $\ml (G) \geq \frac{n}{18} +\frac{13}9$. The reason is that for small values of $\ml (G)$ (at most 8) this gives graphs with a smaller order than the previous construction. Let us consider a 3-connected cubic graph $H$ and put a copy $P^*$ of a vertex-deleted Petersen graph into some of the vertices (see Figure~\ref{fig:gadget_general_3conn}) to obtain the 3-connected graph $G$. Since in all copies $P^*$ there must be either a cubic vertex or a leaf of any spanning tree (again by the non-hamiltonicity of the Petersen graph), by introducing $2k+1$ copies of $P^*$, we obtain at least $k+2$ leaves in any spanning tree. If $H$ has order $2k+2$, the graph $G$   has $9(2k+1) +1 = 18k+10$ vertices, thus $\ml (G) \geq  \frac{n}{18} +\frac{13}9$ indeed. It is worth mentioning that actually $\ml (G) =  \frac{n}{18} +\frac{13}9$, by a theorem in \cite{OZW}. 

\begin{figure}
\begin{center}
\begin{tikzpicture}[thick,every node/.style={fill, circle, inner sep=1pt}]
\begin{scope}
\node (a) at (0,0) {};
\draw (a) -- ++(90:.5) (a) -- ++(210:.5) (a) -- ++(330:.5);
\end{scope}
\begin{scope}[shift={(3.5,0)}]
\node (a) at (90:1) {};
\node (b) at (210:1) {};
\node (c) at (330:1) {};
\node (ab) at ($.75*(a) + .25*(b)$) {};
\node (ba) at ($.75*(b) + .25*(a)$) {};
\node (ac) at ($.75*(a) + .25*(c)$) {};
\node (ca) at ($.75*(c) + .25*(a)$) {};
\node (cb) at ($.75*(c) + .25*(b)$) {};
\node (bc) at ($.75*(b) + .25*(c)$) {};
\draw (a) -- (ab) -- (ba) -- (b) -- (bc) -- (cb) -- (c) -- (ca) -- (ac) -- (a)
      (ab) -- (cb) (ba) -- (ca) (ac) -- (bc)
      (a) -- ++(90:.5) (b) -- ++(210:.5) (c) -- ++(330:.5);
\end{scope}

\draw[-latex] (1,0) -- (2,0);
\end{tikzpicture}
\end{center}
\caption{An operation to construct a non-traceable 3-connected cubic graph from an arbitrary 3-connected cubic graph.}\label{fig:gadget_general_3conn}
\end{figure}
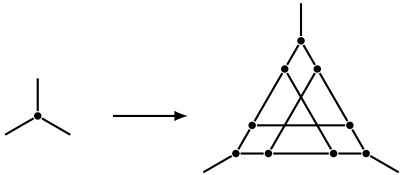

\begin{claim} 
The unique smallest non-traceable (that is, $\ml (G) \geq 3$) 3-connected cubic graph has order 28 and is given by the above construction, for $k=1$ and $H=K_4$. 
\end{claim}

\begin{proof} We have seen that the graph obtained by the above construction with $k=1$ and $H=K_4$ has no spanning tree with fewer than 3 leaves and it obviously has order 28. 

We implemented two independent programs to test if a given graph is traceable. We then used the generator for cubic graphs called \textit{snarkhunter}~\cite{snarkhunter-site, brinkmann_11} to generate all 2-connected cubic graphs up to 32 vertices and tested the generated graphs for traceability. The results of both traceability programs were in complete agreement and showed that the smallest 2-connected non-traceable cubic graphs have 28 vertices and that there is a unique 3-connected non-traceable cubic graph on 28 vertices. The complete counts can be found in Table~\ref{table:2conn}.
\end{proof}

	\begin{table}[htb!]
		\centering 
		 \renewcommand{\arraystretch}{1.1} 
		\begin{tabular}{cccc} 
			\toprule 
			Order & Connectivity 2 & Connectivity 3 & Total  \\ 
			\midrule 
			0-26 & 0 & 0 & 0 \\ 	
			28 & 9 & 1 & 10 \\
			30 & 122 & 9 & 131 \\
			32 & 1814 & 126 & 1940 \\			
			\bottomrule 
		\end{tabular}
		\caption{Counts of all 2-connected non-tracable cubic graphs up to 32 vertices. All of these graphs have $\ml (G) = 3$.}\label{table:2conn}
	\end{table}		
	
The nine non-traceable cubic graphs on 28 vertices with connectivity 2 are shown in Figure~\ref{fig:28_conn2} and the unique non-traceable cubic graph on 28 vertices with connectivity 3 is shown in Figure~\ref{fig:28_conn3}. Finally, the nine non-traceable 3-connected cubic graphs on 30 vertices are shown in Figure~\ref{fig:30_conn3}. 
The graphs from Table~\ref{table:2conn} up to 30 vertices can also be downloaded and inspected in the database of interesting graphs from the \textit{House of Graphs}~\cite{BCGM} by searching for the keywords ``2-connected non-traceable''.

\begin{figure}
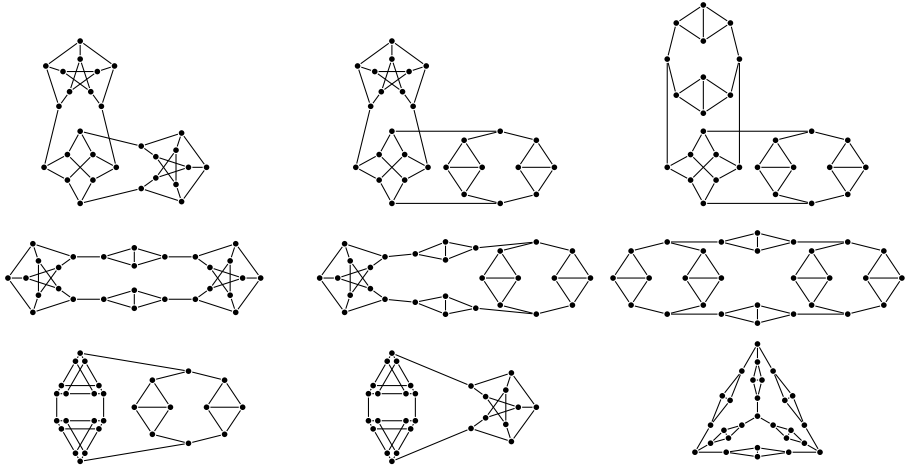

\tikzset{every node/.style={circle, fill, inner sep=.75pt, thick},every edge/.style={draw}}
\newcommand{\imagescale}{.24}
\newcommand{\imagebbL}{\useasboundingbox (-5,-3) rectangle (11,10);}
\newcommand{\imagebbdiamonds}{\useasboundingbox (-8,-3) rectangle (8,3);}
\newcommand{\imagebblastrow}{\useasboundingbox (-5,-4) rectangle (11,4);}
\newcommand{\imagebblastrowcenter}{\useasboundingbox (-8,-3.5) rectangle (8,4.5);}
\begin{center}

\end{center}
\caption{The nine non-traceable cubic graphs on 28 vertices with connectivity 2.}
\label{fig:28_conn2}
\end{figure}

\begin{figure}
\begin{center}
%
\end{center}
\caption{The unique non-traceable cubic graph on 28 vertices with connectivity 3.}
\label{fig:28_conn3}
\end{figure}

\begin{figure}
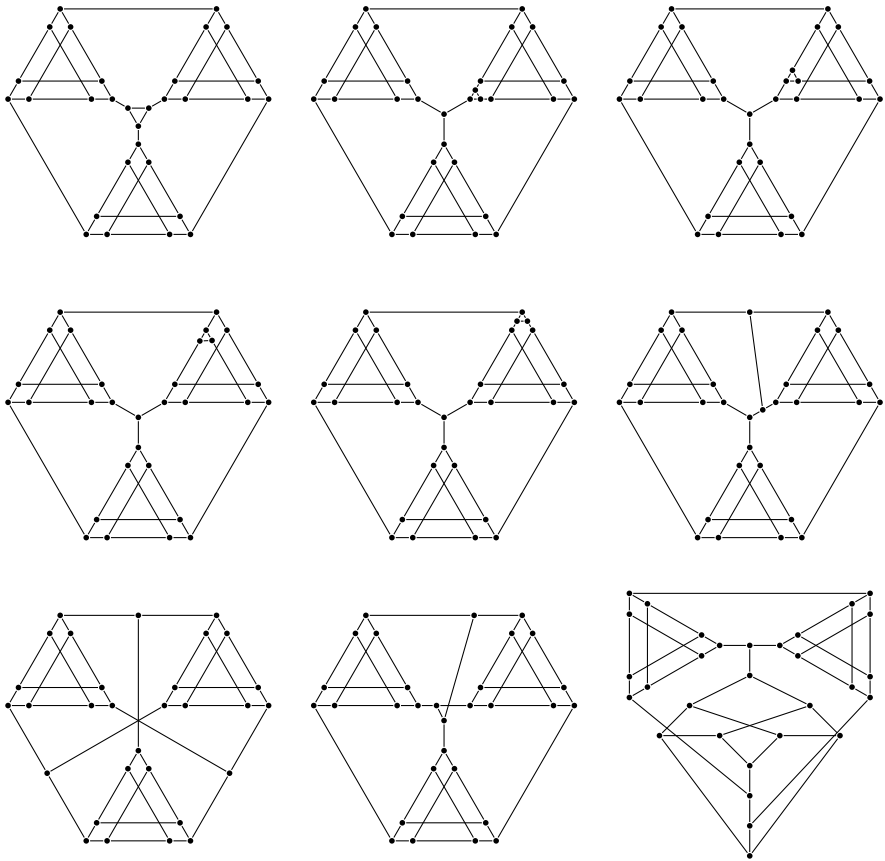

\tikzset{every node/.style={circle, fill, inner sep=.75pt, thick},every edge/.style={draw}}
\newcommand{\imagescale}{.4}
\newcommand{\imagebb}{\useasboundingbox (-5,-5) rectangle (5,5);}
\begin{center}
%
\end{center}
\caption{The nine non-traceable 3-connected cubic graphs on 30 vertices.}
\label{fig:30_conn3}
\end{figure}

\begin{claim} \label{claim_ml3}
All 2-connected non-traceable cubic graphs $G$ on up to 32 vertices have $\ml (G) = 3$, so the smallest 2-connected cubic graphs $G$ with $\ml (G) > 3$ have at least 34 vertices.
\end{claim}
\begin{proof}
We implemented the Mayeda-Seshu algorithm \cite{MS:65} to compute all spanning trees and from this determined the minimum leaf number of all 2-connected non-traceable cubic graphs on up to 32 vertices and all graphs had at least one spanning tree with exactly three leaves.
\end{proof}

As an immediate corollary we get the following.

\begin{corollary}
Conjecture~\ref{sej1} and Conjecture~\ref{sej2} are true for graphs on up to 32 vertices.
\end{corollary}

\begin{remark} 
It is worth mentioning that the smallest member of the first sequence of 3-connected cubic graphs (showing that Conjecture~\ref{sej2} is sharp) has order 40. However, the ceiling sign makes the conjecture still plausible.
\end{remark} 



\subsection{Planar case}

Now we turn our attention to the planar case. The Zoeram-Yaqubi example is planar, so for not necessarily 2-connected planar graphs one has the same results as for the general case. 

\begin{conjecture} \label{sej3}
If $G$ is a 2-connected planar cubic graph  of order $n$ then $\ml (G) \leq \frac{n}{14} + 1 $.
\end{conjecture}

\begin{figure}
\begin{center}
\begin{tikzpicture}[thick,every node/.style={fill, circle, inner sep=1pt}]
\begin{scope}
\node (a) at (-1,0) {};
\node (b) at (0,0) {};
\draw (a) -- (b);
\end{scope}
\begin{scope}[shift={(4,0)}]
\node (a) at (-1,0) {};
\node (b) at (1,0) {};
\node (c) at (-.5,0) {};
\node (d) at (.5,0) {};
\node (e) at (0,.5) {};
\node (f) at (0,-.5) {};
\draw (a) -- (c) -- (e) -- (f) -- (d) -- (b) (c) -- (f) (d) -- (e);
\end{scope}

\draw[-latex] (1,0) -- (2,0);
\end{tikzpicture}
\end{center}
\caption{An operation to construct a non-traceable 2-connected planar cubic graph from an arbitrary 2-connected cubic multigraph.}\label{fig:gadget_planar_2conn}
\end{figure}
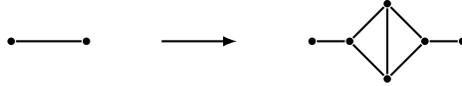

Let $H$ be an arbitrary 2-connected cubic planar multigraph of order $k$ and let us replace every edge of $H$ by an edge-deleted $K_4$ (we delete the edges and connect their endvertices to the degree 2 vertices of the respective copy of the edge-deleted $K_4$, see Figure~\ref{fig:gadget_planar_2conn}). Let us call the graph obtained in this way $G$. It is obvious that $G$ is 2-connected, cubic and its order is $n=7k$. Yu~\cite{Yu} showed that $\mu (G) \geq \frac{n}{14}  $ (actually he proved it for the case when $H$ is a simple graph, but the proof also works for multigraphs). First we give a short proof of this result: by deleting those $k$ vertices of $G$ that also appear in $H$, we obtain $\frac32 k$ components, thus $G$ cannot be covered by less than $\frac32 k - k =\frac{k}2=\frac{n}{14}$ paths. This immediately implies $\ml (G) \leq \frac{n}{14} + 1 $ (by Claim~\ref{all1}). In order to provide some support for Conjecture~\ref{sej3}, we also prove that regardless of the starting graph $H$, we have $\ml (G) = \frac{n}{14} + 1 $, from which $\mu (G) = \frac{n}{14}$ also follows (again, by Claim~\ref{all1}). Therefore these graphs $G$ show that the bound $\mu (G) +1 \leq \ml (G)$ of Claim \ref{all1} is sharp even for 2-connected cubic graphs. 

\begin{claim} \label{allml14}
If \(G\) is a graph obtained by the construction above, then $\ml (G) = \frac{n}{14} + 1 $.
\end{claim}  

\begin{proof} We have seen that $\ml (G) \geq \frac{n}{14} + 1 $, thus it suffices to show that $G$ has a spanning tree with $\frac{n}{14} + 1 = \frac{k}2 +1$ leaves. By Petersen's well-known theorem there exists a perfect matching in $H$ (since it is 2-connected and cubic), thus there also exists a 2-factor $F$, which consists of (say) $m$ cycles. Since the edge-deleted $K_4$'s have a hamiltonian path between the degree 2 vertices, any cycle $C$ of the 2-factor $F$ of $H$ naturally determines a cycle $C'$ in $G$ which goes through all vertices of ($G$ corresponding to) $C$ and all vertices of the copies of the edge-deleted $K_4$'s substituting the edges of $C$. This gives a set of $m$ cycles $F'$ in $G$ covering $5k$ vertices. 

Let us now cover the rest of the vertices (those vertices of $G$ that belong to the edges of the perfect matching of $H$) by $\frac{k}2$ paths (this is again possible because there is a hamiltonian path in the edge-deleted $K_4$ between the degree 2 vertices). Some of these paths can be joined to the cycles of $F'$, more precisely to each cycle of $H$ in $F$ we can assign an edge of the perfect matching of $H$, such that every edge is assigned to at most one cycle
and every assigned edge connects two cycles of $F'$. The reason is that if we contract the cycles, we obtain a 2-connected graph (more precisely, 2-connected, if it has at least 3 vertices, but always 2-edge-connected),
and we just have to orient the edges of the resulting graph $G'$ (if there are any), such that all vertices have in-degree at least 1, which is not difficult. We just have to delete an edge (the graph remains connected), then execute a Breadth First Search from a vertex adjacent with the deleted edge, and orient the edges from the root. The deleted edge is oriented towards the root. If $G'$ has just one vertex (that is, $F$ consists of one cycle), any edge of the perfect matching is suitable.

In this way we obtained a cover of all vertices of $G$ by a set $S$ of $m$ cycles to which paths of 4 vertices are joined and a set of $\frac{k}2- m$ paths of 4 vertices. By deleting a suitable edge of each element of $S$ we obtain a path cover of $G$   
 by $m+(\frac{k}2-m)=\frac{k}2$ paths, such that all paths have an endvertex adjacent to a vertex of another path (this latter property is obvious for the paths of 4 vertices and also easy to see for the longer paths becasue of the choice that the assigned edge connects two cycles of $F'$).
This means that we can extend the forest consisting of the paths to a spanning tree with at most $\frac{k}2 +1$ leaves, and the result follows.
\end{proof}

For 3-connected planar graphs the upper bound is probably much better, since the smallest non-hamiltonian such graph is the Barnette-Bos\'ak-Lederberg graph of order 38~\cite{Bo,Le}. 
For a lower bound we can use the construction given for the 3-connected case (Conjecture~\ref{sej2}), with a J-cell obtained from the Barnette-Bos\'ak-Lederberg graph instead of the Petersen graph. 
This gives a lower bound of $\frac{n}{72}+\frac12$. We can also use the second construction for 3-connected graphs (again with the Barnette-Bos\'ak-Lederberg graph, with one vertex deleted), giving a lower bound $\frac{n}{74}+\frac{55}{37}$, which is not better than $\lceil \frac{n}{72}+\frac12 \rceil$. (The small values are better without the ceiling, up to 36 leaves.)

Knorr~\cite{Kn:10} showed that all 3-connected cubic planar graphs on up to at least 52 vertices are traceable. 
The smallest non-traceable 3-connected cubic planar graph known so far has 88 vertices and was constructed by Zamfirescu~\cite{Za:80}. 
Using the generator \textit{plantri}~\cite{BM07} and our program for testing if a graph is traceable, we were able to show the following.

\begin{claim}
All 3-connected cubic planar graphs with girth 5 up to at least 74 vertices are traceable.
\end{claim}

\subsection{Bipartite case}

Finally we deal with the bipartite case.

\begin{conjecture}
If $G$ is a bipartite cubic graph of order $n$  then $\ml (G) \leq \frac{n}{20}+1$. 
\end{conjecture}
\begin{conjecture}
If $G$ is a bipartite planar cubic graph of order $n$  then $\ml (G) \leq \frac{n}{26}+1$. 
\end{conjecture}

\begin{figure}
\newcommand{\imagescale}{1.25}
\begin{center}
\begin{tikzpicture}[thick,every node/.style={fill, circle, inner sep=1pt},scale=\imagescale]
\useasboundingbox (-2,-1) rectangle (7,1);
\begin{scope}
\node (a) at (-1,0) {};
\node (b) at (0,0) {};
\draw (a) -- (b);
\end{scope}
\begin{scope}[shift={(4,0)}]
\node (a) at (-1,0) {};
\node (b) at (2,0) {};
\node (c) at (-.5,0) {};
\node (d) at (1.5,0) {};
\node (e) at (0,.5) {};
\node (f) at (0,-.5) {};
\node (g) at (1,.5) {};
\node (h) at (1,-.5) {};
\draw (a) -- (c) -- (e) -- (g) -- (d) -- (h) -- (f) -- (c) (e) -- (h) (f) -- (g) (d) -- (b);
\end{scope}

\draw[-latex] (1,0) -- (2,0);
\end{tikzpicture}

\begin{tikzpicture}[thick,every node/.style={fill, circle, inner sep=1pt},scale=\imagescale]
\useasboundingbox (-2,-1) rectangle (7,1);
\begin{scope}
\node (a) at (-1,0) {};
\node (b) at (0,0) {};
\draw (a) -- (b);
\end{scope}
\begin{scope}[shift={(4.25,0)}]
\node (a) at (-1,0) {};
\node (b) at (1.5,0) {};
\node (c) at (-.5,0) {};
\node (d) at (1,0) {};
\node (e) at (0,.5) {};
\node (f) at (0,-.5) {};
\node (g) at (.5,.5) {};
\node (h) at (.5,-.5) {};
\node (i) at (0,0) {};
\node (j) at (.5,0) {};
\draw (a) -- (c) -- (e) -- (g) -- (d) -- (h) -- (f) -- (c) (e) -- (i) -- (f) (g) -- (j) -- (h) (i) -- (j) (d) -- (b);
\end{scope}

\draw[-latex] (1,0) -- (2,0);
\end{tikzpicture}
\end{center}
\caption{Top: An operation to construct a non-traceable 2-connected bipartite cubic graph from an arbitrary 2-connected bipartite cubic graph. Bottom: An operation to construct a non-traceable 2-connected bipartite planar cubic graph from an arbitrary 2-connected bipartite planar cubic graph.}
\label{fig:2conn_bip}
\end{figure}
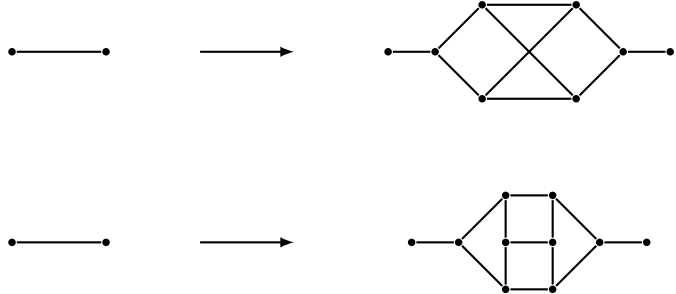

The lower bound is based on the same construction as for 2-connected planar graphs (cf.\ Conjecture~\ref{sej3}), but here $H$ is chosen to be a cubic bipartite graph and instead of the edge-deleted $K_4$ we use an edge-deleted $K_{3,3}$ (see Figure~\ref{fig:2conn_bip}). Note that bipartite connected cubic graphs are also 2-connected. For planar cubic bipartite (and hence 2-connected)  graphs we can use the same construction with an edge-deleted cube (see Figure~\ref{fig:2conn_bip}), giving a lower bound of $\frac{n}{26}+1$.  Finally, planar cubic 3-connected bipartite graphs may not be of much interest here, since by the 
(still open)  Barnette-conjecture these are all hamiltonian.


\section{Proofs of the main results}
\label{sect:proof_main_results}

In this section we prove Theorems~\ref{fo} and~\ref{al}, which we here restate:

\mainthma*
\mainthmb*

In order to prove Theorem~\ref{fo} let us choose a minimal vertex disjoint path (vdp) cover $S$ of $G$. 
By the theorem of Yu~\cite{Yu}, $S$ has size at most $\frac{n}{10}$. A path $P$ is said to be \emph{long}, if it contains at least 18 vertices, otherwise it is \emph{short}. 
The following lemma is about subgraphs of $G$ induced by the vertex set of a short path. 

\begin{lemma} \label{short}
Let $G$ be a traceable graph of order at most 17 where all vertices have degree 2 or 3.
Suppose that the number of degree 2 vertices is at least 2 and for every cut vertex $v$ of $G$, each component  
of $G-v$ contains a degree 2 vertex in $G$. Then there exists a hamiltonian path of $G$ starting
at one of the degree 2 vertices.
\end{lemma}

\begin{proof}
We implemented two independent programs which test if a given graph has the properties from the lemma. We then used the generator \textit{geng}~\cite{nauty-website, mckay_14} to generate all graphs with minimum degree 2 and maximum degree at most 3 up to 17 vertices. Using our two independent programs we verified that among the generated graphs all traceable graphs with at least two degree 2 vertices and for which every cut vertex $v$ of $G$, each component  
of $G-v$ contains a degree 2 vertex, have a hamiltonian path starting at a degree 2 vertex and the results of both programs were in complete agreement.
\end{proof}

Note that this result is sharp in the sense that there are four graphs on 18 vertices which satisfy all the properties of the lemma above, but which do not have a hamiltonian path starting at a degree 2 vertex (see Figure~\ref{fig:no_2_path}). These four graphs can also be downloaded and inspected in the database of interesting graphs from the \textit{House of Graphs}~\cite{BCGM} by searching for the keywords ``no hamiltonian path starting at a degree 2 vertex''.

\begin{figure}
\tikzset{every node/.style={circle, fill, inner sep=1.5pt, thick},every edge/.style={draw, thick}}
\newcommand{\imagescale}{.6}
\begin{center}
\begin{tikzpicture}[scale=\imagescale]
    \useasboundingbox (-2,-3) rectangle (7,3);
    \node (0) at (3.00000000000000,0.000000000000000) {};
    \node (1) at (6.00000000000000,-1.73205080756888) {};
    \node (2) at (6.00000000000000,1.03923048454133) {};
    \node (3) at (-1.00000000000000,1.73205080756888) {};
    \node (4) at (1.40000000000000,-0.346410161513775) {};
    \node (5) at (1.40000000000000,0.346410161513775) {};
    \node (6) at (-1.00000000000000,-1.73205080756888) {};
    \node (7) at (3.60000000000000,0.346410161513775) {};
    \node (8) at (5.40000000000000,-1.38564064605510) {};
    \node (9) at (6.00000000000000,1.73205080756888) {};
    \node (10) at (-1.00000000000000,1.03923048454133) {};
    \node (11) at (-0.400000000000000,1.38564064605510) {};
    \node (12) at (-1.00000000000000,-1.03923048454133) {};
    \node (13) at (-0.400000000000000,-1.38564064605510) {};
    \node (14) at (2.00000000000000,0.000000000000000) {};
    \node (15) at (6.00000000000000,-1.03923048454133) {};
    \node (16) at (5.40000000000000,1.38564064605510) {};
    \node (17) at (3.60000000000000,-0.346410161513775) {};
    \path (0) edge (7);
    \path (0) edge (14);
    \path (0) edge (17);
    \path (1) edge (8);
    \path (1) edge (15);
    \path (2) edge (9);
    \path (2) edge (15);
    \path (2) edge (17);
    \path (3) edge (10);
    \path (3) edge (11);
    \path (4) edge (10);
    \path (4) edge (13);
    \path (4) edge (14);
    \path (5) edge (11);
    \path (5) edge (12);
    \path (5) edge (14);
    \path (6) edge (12);
    \path (6) edge (13);
    \path (7) edge (15);
    \path (7) edge (16);
    \path (8) edge (16);
    \path (8) edge (17);
    \path (9) edge (16);
    \path (10) edge (12);
    \path (11) edge (13);
\end{tikzpicture}\hfill
\begin{tikzpicture}[scale=\imagescale]
    \useasboundingbox (-3,-3) rectangle (6,3);
    \node (0) at (1.00000000000000,1.03923048454133) {};
    \node (1) at (0.400000000000000,-1.38564064605510) {};
    \node (2) at (2.00000000000000,-1.73205080756888) {};
    \node (3) at (2.00000000000000,1.73205080756888) {};
    \node (4) at (4.40000000000000,0.346410161513775) {};
    \node (5) at (-2.00000000000000,0.000000000000000) {};
    \node (6) at (1.00000000000000,-1.03923048454133) {};
    \node (7) at (0.400000000000000,1.38564064605510) {};
    \node (8) at (2.60000000000000,-1.38564064605510) {};
    \node (9) at (2.00000000000000,1.03923048454133) {};
    \node (10) at (5.00000000000000,0.000000000000000) {};
    \node (11) at (-1.40000000000000,-0.346410161513775) {};
    \node (12) at (1.00000000000000,-1.73205080756888) {};
    \node (13) at (-1.40000000000000,0.346410161513775) {};
    \node (14) at (1.00000000000000,1.73205080756888) {};
    \node (15) at (2.60000000000000,1.38564064605510) {};
    \node (16) at (4.40000000000000,-0.346410161513775) {};
    \node (17) at (2.00000000000000,-1.03923048454133) {};
    \path (0) edge (6);
    \path (0) edge (11);
    \path (0) edge (14);
    \path (1) edge (7);
    \path (1) edge (11);
    \path (1) edge (12);
    \path (2) edge (8);
    \path (2) edge (12);
    \path (2) edge (17);
    \path (3) edge (9);
    \path (3) edge (14);
    \path (3) edge (15);
    \path (4) edge (10);
    \path (4) edge (15);
    \path (4) edge (17);
    \path (5) edge (11);
    \path (5) edge (13);
    \path (6) edge (12);
    \path (6) edge (13);
    \path (7) edge (13);
    \path (7) edge (14);
    \path (8) edge (15);
    \path (8) edge (16);
    \path (9) edge (16);
    \path (9) edge (17);
    \path (10) edge (16);
\end{tikzpicture}

\begin{tikzpicture}[scale=\imagescale]
    \useasboundingbox (-4.5,-5.75) rectangle (4.5,2.25);
    \node (0) at (-1.90211303259031,0.618033988749895) {};
    \node (1) at (0.000000000000000,1.00000000000000) {};
    \node (2) at (1.90211303259031,0.618033988749895) {};
    \node (3) at (0.000000000000000,-5.50000000000000) {};
    \node (4) at (0.707106781186548,-3.29289321881345) {};
    \node (5) at (-1.50000000000000,-4.00000000000000) {};
    \node (6) at (-0.951056516295154,0.309016994374947) {};
    \node (7) at (-0.587785252292473,-0.809016994374947) {};
    \node (8) at (1.17557050458495,-1.61803398874989) {};
    \node (9) at (0.707106781186548,-4.70710678118655) {};
    \node (10) at (0.000000000000000,-2.50000000000000) {};
    \node (11) at (0.000000000000000,2.00000000000000) {};
    \node (12) at (0.951056516295154,0.309016994374947) {};
    \node (13) at (0.587785252292473,-0.809016994374947) {};
    \node (14) at (-1.17557050458495,-1.61803398874989) {};
    \node (15) at (-0.707106781186548,-4.70710678118655) {};
    \node (16) at (-0.707106781186548,-3.29289321881345) {};
    \node (17) at (1.50000000000000,-4.00000000000000) {};
    \path (0) edge (6);
    \path (0) edge (11);
    \path (0) edge (14);
    \path (1) edge (7);
    \path (1) edge (11);
    \path (1) edge (13);
    \path (2) edge (8);
    \path (2) edge (11);
    \path (2) edge (12);
    \path (3) edge (9);
    \path (3) edge (15);
    \path (4) edge (10);
    \path (4) edge (15);
    \path (4) edge (17);
    \path (5) edge (14);
    \path (5) edge (15);
    \path (5) edge (16);
    \path (6) edge (12);
    \path (6) edge (13);
    \path (7) edge (12);
    \path (7) edge (14);
    \path (8) edge (13);
    \path (8) edge (17);
    \path (9) edge (16);
    \path (9) edge (17);
    \path (10) edge (16);
\end{tikzpicture}\hfill
\begin{tikzpicture}[scale=\imagescale]
    \useasboundingbox (-4.5,-4) rectangle (4.5,4);
    \node (0) at (0.000000000000000,2.00000000000000) {};
    \node (1) at (2.50000000000000,-2.50000000000000) {};
    \node (2) at (0.000000000000000,-1.50000000000000) {};
    \node (3) at (0.707106781186548,0.707106781186548) {};
    \node (4) at (-2.50000000000000,-2.50000000000000) {};
    \node (5) at (-1.50000000000000,0.000000000000000) {};
    \node (6) at (0.000000000000000,3.00000000000000) {};
    \node (7) at (2.50000000000000,0.000000000000000) {};
    \node (8) at (0.707106781186548,-0.707106781186548) {};
    \node (9) at (0.000000000000000,1.50000000000000) {};
    \node (10) at (-2.50000000000000,2.50000000000000) {};
    \node (11) at (0.000000000000000,-2.00000000000000) {};
    \node (12) at (0.000000000000000,-3.00000000000000) {};
    \node (13) at (2.50000000000000,2.50000000000000) {};
    \node (14) at (-2.50000000000000,0.000000000000000) {};
    \node (15) at (-0.707106781186548,-0.707106781186548) {};
    \node (16) at (-0.707106781186548,0.707106781186548) {};
    \node (17) at (1.50000000000000,0.000000000000000) {};
    \path (0) edge (6);
    \path (0) edge (10);
    \path (0) edge (13);
    \path (1) edge (7);
    \path (1) edge (11);
    \path (1) edge (12);
    \path (2) edge (8);
    \path (2) edge (15);
    \path (3) edge (9);
    \path (3) edge (15);
    \path (3) edge (17);
    \path (4) edge (11);
    \path (4) edge (12);
    \path (4) edge (14);
    \path (5) edge (14);
    \path (5) edge (15);
    \path (5) edge (16);
    \path (6) edge (10);
    \path (6) edge (13);
    \path (7) edge (13);
    \path (7) edge (17);
    \path (8) edge (16);
    \path (8) edge (17);
    \path (9) edge (16);
    \path (10) edge (14);
    \path (11) edge (12);
\end{tikzpicture}
\end{center}
\caption{The four traceable graphs of order 18 where all vertices have degree 2 or 3, and for each cut all components contain at least one of the vertices of degree 2, but which have no hamiltonian path starting at one of the vertices of degree 2.}\label{fig:no_2_path}
\end{figure}
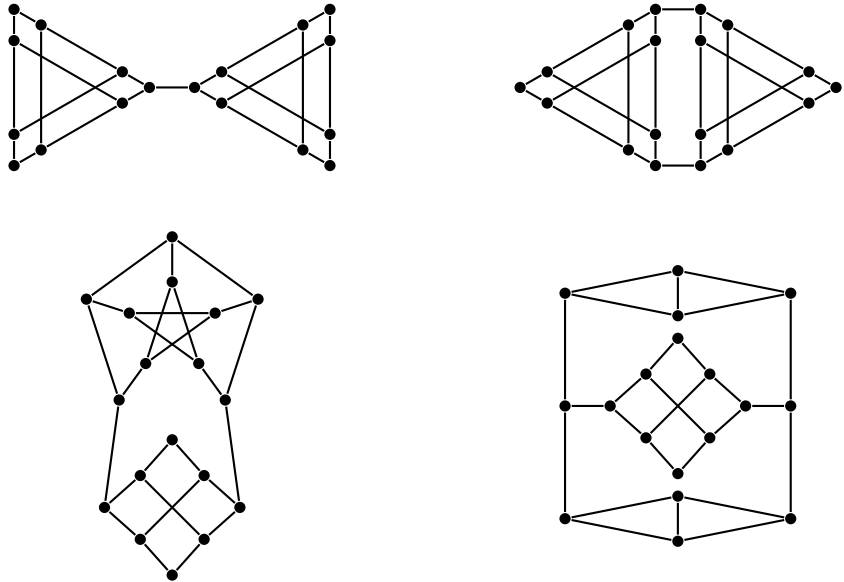

We would like to extend the path cover $S$ to a spanning tree by adding edges between different components of $S$. If we do this  arbitrarily then we obtain a spanning tree in which the number of leaves is at most twice the number of the paths in $S$, that is at most $\frac{n}5 $, which is much worse than what we need.
On the other hand, we claim that by Lemma~\ref{short} a short path $P$ (more precisely a path $P'$ on the vertex set of $P$) can be joined to some other path at the endvertex of $P'$. In order to show this, let us consider the graph $H=G[V(P)]$, that is the subgraph of $G$ spanned by the short path $P$. If any of the endvertices of the path $P$ has degree 1 or 2 in $H$, then $P$ obviously can be joined to some other path at an endvertex. If this is not the case, then let us observe that $H$ fulfills all conditions of Lemma \ref{short}. Indeed: $H$ has vertices of degree 2 and 3 only and the number of degree 2 vertices is at least 2, by the
2-connectivity of $G$. Finally, for every cut vertex $v$ of $H$, each component
of $H-v$ contains a degree 2 vertex, otherwise $v$ would be a cut-vertex of $G$, contradicting the 2-connectivity of $G$, again. Thus by Lemma~\ref{short}, $H$ has a hamiltonian path $P'$ starting at one of the degree 2 vertices $t$, hence $P'$ can be joined to some other path of the vdp cover at $t$.

\smallskip

If we can build a spanning tree, such that all short paths are joined this way, we obtain a spanning tree with at most $s+2\ell$ leaves, where $s$ and $\ell$ are the number of short and long paths, respectively.

First let us check that this would indeed give the desired bound. We know that 
\[s+\ell \leq \frac{n}{10} \; \; \; \; \; \; \; \; \; \hbox{and} \; \; \; \; \; \; \; \; \; s+18\ell \leq n,\] 
where the second inequality is true because short paths contain at least 1, long paths contain at least 18 vertices. Then 
\[ \frac{16}{17} (s+\ell ) \leq \frac{16}{17} \cdot \frac{n}{10} \; \; \; \; \; \; \; \; \; \hbox{and} \; \; \; \; \; \; \; \; \; \frac{1}{17} (s+18\ell) \leq  \frac{n}{17}.\] Taking the sum of these two inequalities we obtain \[ s+2\ell \leq \frac{16+10}{17\cdot 10}n = \frac{13}{85} n.\] 

Next let us observe that a spanning tree with $s+2\ell$ leaves is not so easy to build. Though we may join the short paths to some other path at an endvertex, there is no guarantee that after joining a short path $P_1$ to another short path $P_2$, we still have the chance to also join $P_2$ to some path different from $P_1$. The following lemma is used to handle this problem. The number of vertices of a path $P$ is denoted by $|P|$.

\begin{lemma} 
Let $S$ be a mimimal vdp cover of a graph $G$, such that 
$\sum_{P\in S} |P|^2$ is maximum. 
Let furthermore $P, Q \in S$, such that $|P| \leq |Q|$. Then $Q$ cannot be joined to $P$ at an endvertex of $Q$. 
\end{lemma}

\begin{proof} Let $p=|P|, q=|Q|$.  
If $Q$ can be joined to $P$ at an endvertex of $Q$ we easily obtain two paths $Q'$ and $P'$ covering the same vertices as $Q$ and $P$, such that for some integer $k \geq 1$ we have $|Q'| = q+k$ and $|P'| = p-k$. Then \[|Q'|^2 + |P'|^2 = q^2+p^2 + 2k^2 + 2k(q-p) > q^2 + p^2,\]  
that is, by substituting $P$ and $Q$ in $S$ by $P'$ and $Q'$, we obtain a minimal vdp cover contradicting the choice of $S$. 
\end{proof}
 
By the lemma, a short path can only be joined to a longer path, thus we can obtain the desired spanning tree by choosing a minimal vdp cover $S$, such that $\sum_{P\in S} |P|^2$ is maximum, then by 
joining the short paths one by one in the increasing order of their lengths (and adding some arbitrary edges between components when all the short paths are handled). Therefore the bound $\frac{13}{85} n$ of Theorem~\ref{fo} can be achieved, indeed.


It is worth mentioning a consequence of the above proof.

\begin{thm} \label{utso}
If $G$ is 2-connected cubic graphs of order $n$ then  $\ml (G) = 2\mu (G) $ is possible only if $\mu (G) \leq \frac{n}{18} $.
\end{thm}

\begin{proof}
If $G$ has a vdp cover containing at least one short path, then by the above proof, clearly $\ml (G) < 2\mu (G) $. Thus any vdp cover of $G$ consists of only long paths, from which $\mu (G) \leq \frac{n}{18} $ follows. 
\end{proof}

Now we give a proof of Theorem~\ref{al}. For this, we use the so-called Refined Depth First Search (RDFS), defined in~\cite{sala}. In that paper it is proved (in Theorem 8) that RDFS finds a spanning tree with at least $\frac56 n-\frac43$ non-leaves, that is $\ml (G)  \leq \frac{n}6 + \frac43$. Now we give a slightly improved analysis of RDFS showing that it actually finds a spanning tree with at least $\frac56 n-\frac13$ non-leaves, provided we start the RDFS at a cut-vertex. Actually, in order to do this, we only have to change the last paragraph of the proof of Theorem 8. of~\cite{sala}. If the root $r$ of the RDFS-tree $T$ is a cut-vertex, it has degree at least 2 in $T$, and therefore $r$ can only appear once in the multiset $M=\{ c(l), c'(l), g(l), h(l) :  \hbox{$l$ is a d-leaf of $T$} \}$ and if $r$  appears once, then it also has degree 2 in $T$. Moreover, d-leaves and leaves of $T$ are the same, since $r$ has degree at least 2. Thus $T$ has at least $4v_1$ distinct vertices of degree 2, where $v_i$ denotes the number of vertices of degree $i$ in $T$.

Now the desired bound easily follows for cubic graphs $G$ of connectivity 1: let us just start the RDFS at a cut-vertex, then for the RDFS-tree $T$ obtained $n=v_1+v_2+v_3\geq v_1+4v_1+(v_1-2) = 6v_1-2$, that is $T$ has at most $\frac{n}6 + \frac13$ leaves. If $G$ has connectivity at least 2, then by Theorem~\ref{fo} $G$ has a spanning tree with at most $\frac{13}{85} n$ leaves, and the proof of Theorem~\ref{al} is now complete.

\section*{Acknowledgements}
Most of the computations were carried out using the Stevin Supercomputer Infrastructure at Ghent University.


\end{document}